\documentclass[12pt]{amsart}
\usepackage{amsmath, amssymb, amsthm, mathrsfs} 
\usepackage{hyperref} 
\usepackage{enumitem} 
\usepackage{geometry} 
\usepackage{xcolor}
\usepackage{cleveref}
\usepackage{graphicx} 
\usepackage{tikz}
\usetikzlibrary{decorations.markings}
\usepackage{pgfplots}
\usepackage{subcaption}
\usepackage{etoolbox} 

\AtBeginEnvironment{proof}{\begingroup\setlength{\parindent}{0pt}}
\AtEndEnvironment{proof}{\endgroup}

\pgfplotsset{compat=1.18}
\usepgfplotslibrary{groupplots}

\numberwithin{equation}{section}

\theoremstyle{plain}
\newtheorem{theorem}{Theorem}
\newtheorem{lemma}{Lemma}
\newtheorem{proposition}{Proposition}

\theoremstyle{definition}

\newtheorem{example}{Example}
\newtheorem{remark}{Remark}

\newcommand{\R}{\mathbb{R}}
\newcommand{\C}{\mathbb{C}}

\newcommand{\supp}{\mathrm{supp}\,}
\newcommand{\WF}{\mathrm{WF}_h}

\newcommand{\Order}{\mathcal{O}} 
\newcommand{\PW}{\mathcal{P}_1} 
\newcommand{\QW}{\mathcal{P}_2}
\newcommand{\HW}{\mathcal{P}}

\renewcommand{\Re}{\operatorname{Re}}
\renewcommand{\Im}{\operatorname{Im}}

\begin{document}

\title[Propagation of Semiclassical Wave Front Sets]{Propagation of Semiclassical Wave Front Sets for Non-Self-Adjoint Matrix-Valued Pseudo-Differential Operators}

\author{Yuta Kawamura}




\begin{abstract}
We study the semiclassical wave front set of vector-valued microlocal solutions to the pseudo-differential system with a non-Hermitian matrix-valued symbol.
The main result establishes the propagation of the semiclassical wave front sets under a generalized microhyperbolicity condition for the real part of the symbol and the non-negativity for the imaginary part.
The proof is based on H\"{o}rmander’s positive commutator method.
\end{abstract}


\maketitle
\section{Introduction}

Understanding how singularities propagate is a central problem in the analysis of hyperbolic partial differential equations.
It is well known that the wavefront set of solutions to hyperbolic equations propagates along bicharacteristics, namely the null geodesics of the principal symbol; see Lars H\"ormander~\cite{hormander4}. However, in many physically relevant situations, waves evolve in domains with boundaries, where interactions with the boundary significantly influence their behavior.
In this setting, the propagation of singularities becomes more subtle: for the wave equation, and more generally for strictly hyperbolic operators, singularities follow generalized broken bicharacteristics, undergoing reflections at the boundary in accordance with the laws of geometric optics; see Victor Ivrii and Vesselin Petkov~\cite{IP}.
The case of systems of self-adjoint h-pseudodifferential operators has been extensively studied by Victor Ivrii~\cite{ivrii2019microlocal}.
In the present work, we are concerned with the non-self-adjoint (non-Hermitian) setting, where new phenomena arise and the classical propagation picture requires further refinement.

A function $a(x,\xi;h)$ on the phase space $\R^{2n}=\R_x^n\times\R_\xi^n$ depending on a small positive parameter $h$ is said to be in the symbol class $S(1)$ when it is bounded in $\R^{2n}$ uniformly in $h$ with all its derivatives. The Weyl quantization $a^W$ 
$$  a^W(x,hD;h)u := \frac{1}{(2 \pi h)^n}\iint e^{\frac{i}{h}(x-y) \cdot \xi} a \left( \frac{x+y}{2},\xi ;h\right) u(y) \,dy \,d\xi,
$$
is a bounded operator in $L^2(\R^n)$ (Calder\'on-Vaillantcourt's theorem).

Let $u(x;h) \in L^2(\R^n)$ with $\|u\| \le 1$ for all small $h$.
The semiclassical wave front set $\WF(u)$ is defined as the complementary set of the points
$(x_0,\xi_0) \in \R^{2n}$ such that
there exists a symbol $\chi(x,\xi) \in S(1)$ with $\chi(x_0,\xi_0) = 1$ satisfying $\|\chi^W(x,hD)u\| = \Order(h^\infty)$
as $h \to 0$. The set $\WF(u)$ is closed.

We assume that $p$ is of the form
$p(x,\xi;h) = p_0(x,\xi) + h r(x,\xi;h)$,
with $p_0(x,\xi)$, $r(x,\xi;h)\in S(1)$, and that $p_0(x,\xi)$ is of {\it real principal type} at a point $(x_0,\xi_0)$, that is,
$p_0$ is real-valued and satisfies
\begin{equation} \label{rpeq}
  p_0(x_0,\xi_0)=0, \qquad \partial p_0(x_0,\xi_0) \ne 0.
\end{equation}
The Hamilton flow of $p_0$ is the integral curve of the Hamilton vector field $H_{p_0}=
\partial_{\xi}p \cdot \partial_x - \partial_x p \cdot \partial_{\xi}
    \simeq \left( \partial_ \xi p, -\partial_ x p\right)$.
    
 Assume $p^Wu=0$, or more generally $\WF(p^Wu)=\emptyset$ in a neighborhood $U$ of $(x_0,\xi_0)$. Then the propagation theorem tells us that $\WF(u)$ is invariant along the Hamilton flow in $U$ passing through $(x_0,\xi_0)$.
In other words, $(x_0,\xi_0)\notin \WF(u)$ if either $\WF(u)\cup U_+=\emptyset$ or $\WF(u)\cup U_-=\emptyset$, where $U_\pm
=\{(x,\xi);H_{p_0}(x_0,\xi_0)\cdot(x,\xi)\gtrless H_{p_0}(x_0,\xi_0)\cdot(x_0,\xi_0)\}$.  

This theorem has been proved by various methods.
See for example Maciej Zworski~\cite{zw} for the reduction method to the normal form $p_0=\xi_1$, and André Martinez~\cite{martinez2002introduction} for the approach using the FBI transform.

The theorem was generalized to complex-valued symbol $p_0$. Semyon Dyatlov~\cite{Dyatlov2011PROPAGATIONOS} (see also Semyon Dyatlov and Maciej Zworski~\cite{dyatlov2019mathematical}) proved, under the condition that 
$\Re p_0$ is of real principal type and $\Im p_0$ is nonnegative,  that $\WF(u)\cup U_+=\emptyset$ implies $(x_0,\xi_0)\notin \WF(u)$ (but $\WF(u)\cup U_-=\emptyset$ does not imply $(x_0,\xi_0)\notin \WF(u)$).
The proof is based on H\"ormander's positive commutator method and uses the construction of a special escape function that grows exponentially along the Hamiltonian flow of $\Re p_0$ in a neighborhood of $(x_0,\xi_0)$.

On the other hand, Victor Ivrii~\cite{ivrii2019microlocal} generalized the propagation theorem to Hermitian matrix-valued symbols $P(x,\xi;h) = (p_{i,j}(x,\xi;h))_{1 \leq i,j \leq N}$.
Let $P(x,\xi;h) \in S(1)$ be of the form
\begin{equation}\label{H_0}
  P(x,\xi;h) = P_0(x,\xi) + hR(x,\xi;h)
\end{equation}
with $P_0(x,\xi) \in S(1)$ independent of $h$ and $R(x,\xi;h) \in S(1)$, and suppose $\det P_0(x_0,\xi_0)=0$.

The real principal type condition in the scalar case is replaced with the (what he calls) {\it microhyerbolicity condition} at a point $(x_0,\xi_0)$  in a direction $(\xi^*,-x^*)\in\R^{2n}\setminus (0,0)$,
which means that for $g(x,\xi)=x^*\cdot x+\xi^*\cdot\xi$, one has
\begin{equation}\label{eq:ivrii-mch}
\left\langle \{P_0, g I_N\}|_{(x_0,\xi_0)} \omega, \omega \right\rangle  \geq \frac{1}{C} \left\lVert \omega  \right\rVert ^2 - C  \left\lVert P_0(x_0,\xi_0) \omega  \right\rVert^2  \quad
    \text{for all }\omega \in \C^N.
\end{equation}
Remark that the real principal type condition \eqref{rpeq} at $(x_0,\xi_0)$ in the scalar case implies the microhyperbolicity at the same point in the direction $H_{p_0}(x_0,\xi_0)$.
Under the condition \eqref{eq:ivrii-mch}, he proved that $(x_0,\xi_0)\notin \WF(u)$ if $\WF (P^Wu)\cap U=\emptyset$ and either $\WF(u)\cup U_+=\emptyset$ or $\WF(u)\cup U_-=\emptyset$, where
$U_\pm=\{(x,\xi)\in U; g(x,\xi)\gtrless g(x_0,\xi_0)\}.$
Here  $\WF(u)=\bigcup_{k=1}^N \WF(u_k)$ for $u={}^t(u_1,\ldots,u_N)$.

In this article, we further generalize the result of Ivrii to non-Hermitain  matrix-valued symbols employing the positive commutation method that Dyatlov used for the scalar case.
We assume a generalized microhyperbolicity condition \eqref{mcheq}, which is similar to \eqref{eq:ivrii-mch} but $P_0$ on the left hand side is replaced with $\Re P_0=(P_0+P_0^*)/2$. Together with the  non-negativity of the imaginary part $\Im P_0=(P_0-P_0^*)/(2i)$, we prove that $\WF(u)\cup U_+=\emptyset$ implies $(x_0,\xi_0)\notin \WF(u)$ as in the result of Dyatlov.

The key of the proof consists in the local construction of an exponentially growing scalar escape function $g(x,\xi)$
near the point $(x_0,\xi_0)$ for the matrix-valued symbol $\Re P_0(x,\xi)$.
Compared with the scalar case, this construction is more involved,
because it is necessary to construct a common escape function for all the zero eigenvalues of the symbol (see Lemma~\ref{lem1}). 

Microhyperbolicity~\eqref{eq:ivrii-mch} or its non-Hermitian version~\eqref{mcheq} is relevant when the characteristic set $\{ (x,\xi) \in \R^{2n}| \det \Re P_0(x,\xi) = 0\}$ has variable multiplicity, that is, when the zero set of eigenvalues of $\Re P_0$ cross.
Such a situation of variable multiplicity has first been considered for hyperbolic systems in the context of $C^\infty$ well-posedness of Cauchy problems and scattering problems (see Victor Ivrii and Vesselin Petkov~\cite{IP} and  Vesselin Petkov~\cite{P} and references therein).
In the semiclassical setting, the microhyperbolicity condition appears in the asymptotics of the eigenvalue counting function or spectral shift function (see Marouane Assal, Mouez Dimassi and Setsuro Fujiié~\cite{10.1093/imrn/rnx149} and Mouez Dimassi and Johannes Sj\"ostrand~\cite{Dimassi_Sjostrand_1999}).
In the one-dimentional case, in particular, the precise asymptotic distribution of resonances was studied in the presence of crossings of hamiltonian flows (see Marouane Assal, Setsuro Fujiié and Kenta Higuchi~\cite{10.1093/imrn/rnad290} and Setsuro Fujiié, André Martinez and Takuya Watanabe~\cite{published_papers/42164222}).
There, the so-called microlocal scattering matrix at each crossing point plays an essential role.
The definition of this matrix is based on the propagation of singularities at crossing points.
Our theorem will be a starting point of this theory in multidimentional setting.

This paper is organized as follows.
In Section~2, we state the main theorem and  
give several illustrative examples.
Section~3 is devoted to the proof of the main theorem.  We first prove Lemma~\ref{lem1}, which guarantees the existence of an escape function with required support properties. Using this function, we establish a propagation estimate (Proposition~\ref{propest}), which directly proves the main theorem.

\section{Main Results}

\subsection{Assumptions and main result}

In this subsection, we present the main theorem of this paper.

Let $(x_0,\xi_0)\in \R^{2n}, \; (x^*,\xi^*) \in \R^{2n}\setminus \{0\}$ and
 $g(x,\xi) = (x^*,\xi^*) \cdot (x,\xi).$


Suppose $P \in S(1)$ is of the form \eqref{H_0}.
We assume that there exists a constant $C>0$ such that, for all $\omega \in \C^N$, one has
\begin{equation}\label{mcheq}
\left\langle \{\Re P_0, g I_N\}|_{(x_0,\xi_0)} \omega, \omega \right\rangle  \geq \frac{1}{C} \left\lVert \omega  \right\rVert ^2 - C  \left\lVert P_0(x_0,\xi_0) \omega  \right\rVert^2.
\end{equation}
Here $\{A,B\} = \sum \partial_{\xi_j} A \, \partial_{x_j} B - \partial_{x_j} A \, \partial_{\xi_j} B$ is the Poisson bracket. 
We call a scalar function $g$ satisfying~\eqref{mcheq} an escape function.



In the Hermitian case, the condition~\eqref{mcheq} is equivalent to the microyperbolicity~\eqref{eq:ivrii-mch} in the direction $H_g$.
If $N=1$ moreover, \eqref{mcheq} reduces to either the ellipticity $p_0(x_0,\xi_0) \neq 0$ or the real principal type condition~\eqref{rpeq}.
\begin{remark}\label{rem;3}
Assume $\Im P_0 \geq O$. If $\Re P_0$ is microhyperbolic in the direction $H_g$, then $P_0$ satisfies \eqref{mcheq} for this $g$.\begin{footnote}{
The condition $\Im P_0 \geq O$ implies  $C\|P_0 \, \omega\| \geq \| (\Re P_0) \, \omega \|$.
 In fact, let $W = \{ \omega \in \C^{N}; P_0\omega = 0 \}$. 
If $\omega \in W$, then $\langle (\Im P_0)\omega, \omega \rangle = 0$, since $\langle P_0\omega, \omega \rangle = 0$.
From $\|(\Im P_0)^{1/2}\omega\| = 0$, we obtain $(\Im P_0)\omega = 0$, and hence $(\Re P_0)\omega = 0$.
If $\omega \notin W$, there exists $\omega_1 \in W^\perp$ such that $P_0\omega = P_0\omega_1 \ne 0$.
Since $\inf\{ \|P_0\omega_1\|;\, \omega_1 \in W^\perp, \|\omega_1\| = 1 \} > 0$, the inequality follows.}\end{footnote} 
However, the converse is not true.
For exapmle, in the case $n = 1$,
$P_0 = \text{diag} (\xi,i)$ satisfies \eqref{mcheq} for $g(x,\xi) = x$, but $\Re P_0 = \text{diag}(\xi,0)$ is not microhyperbolic in any direction.
\end{remark}

\begin{theorem}\label{PoS1}
Suppose that $P \in S(1)$ is of the form \eqref{H_0} and satisfies \eqref{mcheq},
and let $u\in L^2(\R^n;\C^N)$ with $\|u\| \leq 1$.
Assume moreover that there exists a neighborhood $U$ of $(x_0,\xi_0)$ such that
$\Im P_0 \geq O$ in $U$ and
\begin{align}
\WF (P^W u) & \cap U  = \emptyset, \label{wfhu} \\ 
\WF(u) & \cap  U_+ = \emptyset, \label{wvu+}
\end{align}
where $U_+ = \{(x,\xi) \in U ; g(x,\xi) > g(x_0,\xi_0)\}$.
Then
\[
(x_0,\xi_0)\notin \WF(u).
\]
\end{theorem}


\begin{remark}
Theorem~\ref{PoS1} has already been proved in two special cases:
when $N = 1$~\cite{Dyatlov2011PROPAGATIONOS} \cite{dyatlov2019mathematical},
and when $H_0$ is Hermitian~\cite{ivrii2019microlocal}.
The present theorem generalizes both of these results.
\end{remark}

\begin{remark}
Let $U_- = \{(x,\xi) \in U ; g(x,\xi) < g(x_0,\xi_0) \}$.
If we assume $\Im P_0 \le O$ and 
\begin{equation}
  \WF(u)\cap U_-
= \emptyset, \label{wvu-}
\end{equation}
instead of $\Im P_0 \le O$ and~\eqref{wvu+}, Theorem~\ref{PoS1} still holds true.
In general, however, Theorem~\ref{PoS1} does not remain valid if one assumes $\Im P_0 \ge O$ together with~\eqref{wvu-}, or $\Im P_0 \le O$ together with~\eqref{wvu+}.
\end{remark}

\subsection{Several examples}

In this section, we present several examples related to the theorem.
In Example~\ref{prop3}, we consider a already diagonalized symbol examine conditions equivalent to~\eqref{mcheq}.
Example~\ref{nonsmooth1}  deals with symbols that are not smoothly diagonalizable but nevertheless satisfy the condition~\eqref{mcheq}.

\begin{example} \label{prop3}
Let $P_1 = \text{diag}(p_1, p_2, \ldots, p_k)$, where $p_j(x_0,\xi_0) = 0$ for $j = 1,2,\ldots,k$,
and let $P_2$ be an arbitrary $(N-k)\times(N-k)$ matrix such that $\det P_2(x_0,\xi_0) \neq 0$,
and define
\[
P_0 =
\begin{pmatrix}
P_1 & O_{k,\,N-k} \\
O_{N-k,\,k} & P_2
\end{pmatrix}.
\]
  If $P_0$ satisfies \eqref{mcheq} at $(x_0,\xi_0)$ for $g(x,\xi) = (x^*,\xi^*)\cdot (x,\xi)$, then one has
  \begin{equation} \label{realprincipal}
  H_{ \Re  p_j} g (x_0,\xi_0) > 0, 
  \quad (j = 1,2,\ldots,k).  
  \end{equation}
Let $\gamma_j(t) = \exp{tH_{\Re  p_j}}(x_0,\xi_0)$ be a Hamiltonian flow with initial value $\gamma_j(0) = (x_0,\xi_0)$.
Then the equation \eqref{realprincipal}  means that
$$
\gamma_j( \pm t) \in U_{\pm} = \{ (x,\xi) \in U; g(x,\xi) \gtrless g(x_0,\xi_0) \}  \quad (j  = 1,2, \ldots , k),
$$
for sufficiently small $t > 0$ and some neighborhood $U$ of $(x_0,\xi_0)$.
In other word, the function $g(x,\xi)$ is increasing along all of Hamilton flows $\gamma_j(t)$.
\end{example}

\begin{remark}
In the case $\det P_2(x_0,\xi_0) \neq 0$, the imaginary part of $P_2$ is not required to be nonnegative in order for Theorem~\ref{PoS1} to hold.
This is because, the wave front set can appear only on the zero set of the eigenvalues,
therefore the conclusion of Theorem~\ref{PoS1} follows immediately under this assumption.
\end{remark}


The followings are examples of a symbol $P_0$ with \eqref{mcheq} that is not smoothly diagonalized.

\begin{example}\label{nonsmooth1}
In the case $n=1$, the symbol
\[
P_{0}
=
\begin{pmatrix}
x & \xi \\
\xi & -x
\end{pmatrix}
+ k x I_{2}, \qquad (k>1),
\]
satisfies \eqref{mcheq} at $(x_{0},\xi_{0})=(0,0)$ for $g=-\xi$.
The matrix $P_{0}$ has eigenvalues
$
p_{\pm}=k x \pm \sqrt{x^{2}+\xi^{2}},
$
which are not smooth at the origin.
However, on the characteristic set near the origin $ p_{\pm}^{-1}(0)
=\{\, (x,\xi)\in \mathbb{R}^{2}: \sqrt{k^{2}-1}\,x \pm |\xi| = 0 \,\}\setminus \{0\}$, the function $g(x,\xi)$ increases along the Hamiltonian flows $\gamma_{\pm}$ which lie in the set $p_{\pm}^{-1}(0)$, as in Example~1 (see Figure~\ref{fig:one}).

If $n=2$, the symbol
\[
P_{0}
= |\xi|^{2} I_2 +
\begin{pmatrix}
x_{1} & x_{2} \\
x_{2} & -x_{1}
\end{pmatrix}
+ k x_{1} I_{2}, \qquad (k>1),
\]
satisfies \eqref{mcheq} at the point $(x_{0},\xi_{0})=(0,0)$ for $g=-\xi_{1}$.
The eigenvalues of $P_{0}$ are given by
$
p_{\pm} = |\xi|^{2} + k x_{1} \pm |x|,
$
which are not smooth on $\{ x=0 \}$.
On the set $\{\, (x,\xi)\in \mathbb{R}^{4} : \xi_{1}^{2} + (k \mp 1)x_{1} = 0,\quad x_{2}=\xi_{2}=0 \,\} \setminus \{0\}$,
the Hamiltonian flows $\gamma_{\pm}$ are directed toward the origin, and the function $g$ increases along this direction (see Figure~\ref{fig:two}).
\end{example}


\begin{figure}[htbp]
\centering
\includegraphics[width=0.45\linewidth]{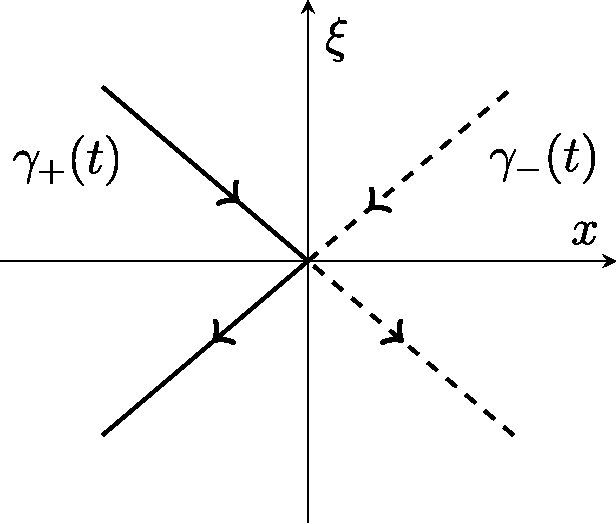}
\caption{}
\label{fig:one}
\end{figure}

\begin{figure}[htbp]
\centering
\includegraphics[width=0.45\linewidth]{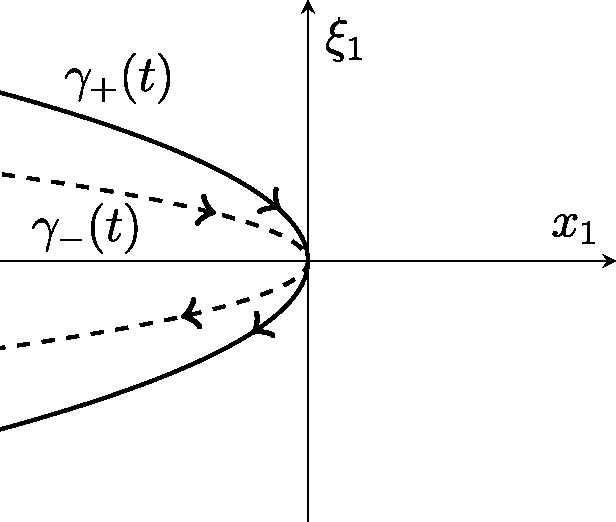}
\caption{}
\label{fig:two}
\end{figure}

\section{proof}


In this section, we prove Theorem~\ref{PoS1}.
Since Theorem~\ref{PoS1} follows immediately from Proposition~\ref{propest} stated below,
we focus on proving that proposition.
To prove Proposition~\ref{propest}, we use Remark~\ref{rem1} which is a slight extension of  the following lemma.

\begin{lemma}\label{lem1}
Let $P_0(x,\xi) \in S(1)$ and $g(x,\xi) = (x^*,\xi^*) \cdot (x,\xi)$ satisfy \eqref{mcheq}.
Then there exists a neighborhood $U$ of $(x_0,\xi_0)$ such that, for any constant $M > 0$, there exist functions
$f(x,\xi), \hat f(x,\xi) \in C_0^\infty(\mathbb{R}^{2n}; \mathbb{R})$ and a compact set $\Omega$ 
for which the followings hold:
\begin{enumerate}[label=\textbf{($C$\arabic*):},ref=($C$\arabic*),align=left, leftmargin=*]
  \item\label{item;C1} $f(x_0,\xi_0) > 0$ and $f \ge 0$ everywhere;
  \item\label{item;C2} $\supp \hat f \cup \supp f \subset U$;
  \item\label{item;Co} $ \Omega \subset U_+ := \{(x,\xi) \in U; g(x,\xi) > g(x_0,\xi_0) \}$;
  \item\label{item;C3} there exists a constant $C > 0$ such that $\{\Re P_0, f I_N\} + C\,\hat f\, P_0^* P_0 \ge M f I_N$ on the complement of $\Omega$.
\end{enumerate}
\end{lemma}

\begin{proof}[Proof of Lemma \ref{lem1}]
We write $z = (x,\xi), \; z^* = (x^*,\xi^*), \; z_0 = (x_0,\xi_0) \in \R^{2n}$.
By a suitable linear transformation, we may assume $z^* = (1,1,\ldots,1)$ without loss of generality. 
Let
\begin{equation}\label{deflam}
g_k(z)
= z_k^* \cdot z,
\quad z_k^* = z^* - \delta e_k, \quad k = 1,2,\ldots,2n,  
\end{equation}
where $\{e_k\}_{k=1}^{2n}$ denotes the standard basis of $\R^{2n}$.
Then, for sufficiently small $\delta > 0$, $\{z_k^*\}_{k=1}^{2n}$ is a system on $\R^{2n}$, and, from~\eqref{mcheq}, we can take sufficiently small neighborhood $U$ of $z_0$ such that  
\begin{equation}\label{mch2}
\{\Re P_0, g_k I_N\}(z)
+ C (P_0^* P_0)(z)
\ge C^{-1} I_N
\end{equation}
uniformly in $U$ for sufficiently large $C>0$ for any $k = 1,2,\ldots,2n$.
By linear independence,
the family $\{g_k\}_{k=1}^{2n}$ may be regarded as a system of independent coordinates on $\R^{2n}$.

For $c,\varepsilon>0$ and $M'=\frac{CM}{2n}>0$,
we choose $\phi_{M'} \in C_0^\infty(\R;\R_{\geq 0})$ such that
\begin{enumerate}[label=\textbf{($\Phi$\arabic*):},ref=($\Phi$\arabic*),align=left, leftmargin=*]
  \item\label{item;phi1} $\supp \phi_{M'} = [- \varepsilon, c + \varepsilon ]$, especially $\phi_{M'}(0) > 0$;
  \item\label{item;phi2} $\partial_t \phi_{M'}(t) \geq M' \phi_{M'}(t)$ on $[c, c + \varepsilon ]^c$. 
\end{enumerate}
We can make $\phi_{M'}(t)$ behave like $e^{-\frac{M'(c + \varepsilon)^2}{t + \varepsilon}}$ on $\left[-\varepsilon, c \right]$.

For $\{g_k\}$ in \eqref{deflam}, we set
\begin{equation} \label{deff}
  f(z) = \prod_{k=1}^{2n} \phi_{M'}\left( g_k(z) - g_k(z_0) \right), \quad
  \hat{f}(z) = \sum_{k=1}^{2n} \frac{\partial f(z)}{\partial g_k}.
\end{equation}

By~\ref{item;phi1} and the independence of $\{g_k\}$, $f$ is compactly supported in
$$
\supp f = \bigcap_k \{z \in \R^{2n};\ -  \varepsilon \le g_k(z) - g_k(z_0) \le c + \varepsilon \},
$$
and hence condition~\ref{item;C1} is trivially satisfied.

It follows from \ref{item;phi2} that
\begin{equation}\label{Mf}
\frac{\partial f(z)}{\partial g_k}  \geq  M'  f(z)
\end{equation}
on the complement of $\Omega := \supp f \setminus \bigcap_k \{z \in \R^{2n}; \; -  \varepsilon \leq g_k(z) - g_k(z_0) <  c \}$. 

Since the vectors $\{z^*, z_1^*, z_2^*, \ldots, z_{2n}^*\}$ are linearly dependent,
the set $\Omega$ is contained in $\{ g(z) > g(z_0) \}$,
provided that $c/\varepsilon$ is chosen sufficiently large and $\delta > 0$ is taken sufficiently small.
In particular, in the case~\eqref{deflam}, it suffices to take $c/\varepsilon > 2n-1$.
Furthermore, if $c>0$ is chosen sufficiently small, then conditions~\ref{item;C2} and~\ref{item;Co} hold: $\supp f \subset U, \; \Omega \subset U_+$.

Finally, we check the condition~\ref{item;C3}.
Using \eqref{mch2} and \eqref{Mf}, and observing that, by the chain rule,
$$
\{\Re P_0,f I_N\} = \sum_{k=1}^{2n} \frac{\partial f(z)}{\partial g_k} \{\Re P_0, g_k I_N \},
$$
we obtain 
$$
\{\Re P_0,f I_N\} + C \hat{f} P_0^*P_0
= \sum_k \frac{\partial f(z)}{\partial g_k} \big(\{\Re P_0, g_k I_N \} + C P_0^* P_0 \big) 
\geq M f I_N
$$
on the complement of $\Omega$.
\end{proof}

By Lemma~\ref{lem1} and its proof, we obtain the following remark.
\begin{remark}\label{rem1}
{\it There exists a neighborhood $U$ of $(x_0,\xi_0)$ such that, for any constant $M > 0$,
there exist functions $\{f_l(x,\xi)\}_{l=0}^{\infty}$ and $\{\hat f_l(x,\xi)\}_{l=0}^{\infty}$ in
$C_0^\infty(\mathbb{R}^{2n}; \mathbb{R})$
and compact sets $\{\Omega_l\}_{l=0}^{\infty}$
for which the following properties hold:
\begin{enumerate}[label=\textbf{($D$\arabic*):},ref=($D$\arabic*),leftmargin=*]
  \item\label{item;D1} $f_0(x_0,\xi_0) > 0$, and $f_0 \ge 0$ everywhere;
  \item\label{item;D2} $\supp f_l \cup \supp \hat f_l \subset \{ f_{l+1} > 0 \} \subset U$;
  \item\label{item;D3} $\overline{\bigcup_l \Omega_l} \subset U_+$;
  \item\label{item;D4} there exists a constant $C > 0$ such that
  $\{\Re P_0, f_l I_N\} + C\,\hat f_l P_0^* P_0 \ge M f_l I_N$
  on the complement of $\Omega_l$.
\end{enumerate}}

Here we describe the key idea of the proof of the above claim.
The families $\{f_l\}$ and $\{\Omega_l\}$ can be constructed as follows, similarly to~\eqref{deff}:
\begin{align*}
  &f_l(x,\xi)
= \prod_{k=1}^{2n}
\phi_{M_l'}\!\left(
\frac{g_k(x,\xi)-g_k(x_0,\xi_0)}{\alpha_l}
\right), \\
&\Omega_l = \supp f_l \setminus \bigcap_k
\left\{ -\varepsilon \le \frac{g_k(x,\xi)-g_k(x_0,\xi_0)}{\alpha_l} < c \right\}.
\end{align*}
Here, $\{\alpha_l\}_{l\ge 0}$ is an increasing sequence of positive numbers satisfying
$\alpha_l \to 1$ as $l \to \infty$ and $M_l' = \alpha_l M'$.
Set $\hat f_l$ as \eqref{deff}, then
conditions~\ref{item;D1}, \ref{item;D2}, and~\ref{item;D4} follow immediately.
To ensure~\ref{item;D3}, noting that
$
\overline{\bigcup_l \Omega_l} = \bigcup_l \Omega_l \cup \Omega,
$
it suffices to check that $\Omega_l \subset U_+$ for all $l$,
which can be done as $\Omega \subset U_+$ in the proof of the lemma.
\end{remark}


The following proposition is called \textit{propagation estimate} (see \cite{dyatlov2019mathematical}).

\begin{proposition}\label{propest}
Assume that $P \in S(1)$ is of the form~\eqref{H_0} and satisfies~\eqref{mcheq}, and let $U$ and $f_0$ be as in Remark~\ref{rem1}.
Moreover, assume that $\Im P_0 \ge O$ in $U$.
Then there exist functions $a, b \in C_0^\infty(\R^{2n})$ such that
$ \supp a \subset U_+, \;
  \supp b \subset U$,
and
\begin{equation}\label{proest}
  \|f_0^W u\| \leq C\|a^W u\| + Ch^{-1}\| b^W P^W u \| + \Order (h^\infty)\|u\|,
\end{equation}
for $C>0$ sufficiently large.
\end{proposition}


By the assumptions~\eqref{wfhu} and~\eqref{wvu+} of the theorem,  together with the assumptions of this proposition, the first and second terms on the right-hand side of the inequality~\eqref{proest} can both be estimated by $\mathcal{O}(h^\infty)\|u\|$.
Since $f(x_0,\xi_0) > 0$, we conclude that $(x_0,\xi_0) \notin \WF(u)$ and complete the proof of Theorem~\ref{PoS1}.

\begin{proof}[Proof of Proposition~\ref{propest}]

Let $\{f_l\}$ and $\{\Omega_l\}$ be as in Remark~\ref{rem1}.
The constant $M > 0$ in Remark~\ref{rem1} will be chosen later.
We choose $\{a_l\} \subset C_0^\infty(\R^{2n})$ such that
\begin{enumerate}[label=\textbf{($A$0):}, ref=($A$0), align=left, leftmargin=*]
\item\label{item;a0}$a_l \equiv 1$ on $\Omega_l$, and $\supp a_l \subset U_+ \cap \{f_{l + 1} \ne 0\}$.
\end{enumerate}

Moreover, we choose symbols $a, b \in C_0^\infty(\R^{2n})$ such that
\begin{enumerate}[label=\textbf{($A$\arabic*):}, ref=($A$\arabic*), align=left, leftmargin=*]
  \item\label{item;a1} $a \equiv 1$ on $\bigcup_l \supp a_l$, and $\supp a \subset U_+$;
  \item\label{item;a2} $b \equiv 1$ on  $\bigcup_l \supp f_l$, and $\supp b \subset U$.
\end{enumerate}
By~\ref{item;D3}, the existence of such symbols $a$ and $\{a_l\}$ is guaranteed.
Moreover, by Remark~\ref{rem1}, for the function $f$ in Lemma~\ref{lem1},
we have $\supp f_l \subset \{ f \neq 0 \}$.
The existence of $b$ follows from~\ref{item;C2}.

We first assume that, for any $l \geq 0$ and any $u \in L^2$, we have
\begin{equation}\label{eq;goal}
\|f_l^W u\| \leq C(\|a_l^W u\| + h^{1/2} \|f_{l+1}^W u\| + h^{-1}\|f_{l+1}^W \HW u\|) + \Order(h^\infty)\|u\|,
\end{equation}
where $\HW = P^W$ and  $\|\cdot\|$ denotes the $L^2$-norm.
Then, for any $l$, there exists $C > 0$ such that
\begin{equation}\label{eq;2}
  \begin{split}
    \|f_0^W u\| \leq &C \sum_{m = 0}^{l} (h^{m/2} \|a_{m}^W u\| + h^{m/2-1} \|f_{m+1}^W \HW u\| ) 
    \\ &+ C h^{(l+1) /2} \|f_{l+1}^W u\| + \Order(h^\infty)\|u\|.
  \end{split}
\end{equation}

From \ref{item;a1} and \ref{item;a2}, 
we obtain that, modulo $\Order(h^\infty) \|u\|$,
$$
\|a_l^W u\| \leq C \|a^W u\|, \qquad \|f_l^W u\| \leq C \|b^W u\| 
$$
and it follows from~\eqref{eq;2} that for any $l \geq 0$, there exist $C > 0$ such that
\begin{equation}
\|f_0^W u\| \leq C(\|a^W u\| + h^{(l + 1)/2} \|b^W u\| + h^{-1}\|b^W \HW u\|) + \Order(h^\infty)\|u\|.
\end{equation}
This implies Proposition~\ref{propest}.

\medskip

Next, we prove the inequality~\eqref{eq;goal}.

We may assume $l=0$ without loss of generality.
We write
\[
P_1 = \Re P, \quad P_2 = \Im P,  \quad \PW= P_1^W,  \quad \QW = P_2^W.
\]
We now compute, denoting by $\langle \cdot, \cdot \rangle$ the $L^2$ inner product,
\begin{equation}
\begin{aligned}\label{P_0Rdcmp}
\Im \left\langle \HW u, (f_0^W)^2 u \right\rangle
&= \Im \left\langle \PW u, (f_0^W)^2 u \right\rangle 
 + \Re \left\langle \QW u, (f_0^W)^2 u \right\rangle \\
&= \frac{1}{2i}\left\langle [(f_0^W)^2 I_N, \PW] u, u \right\rangle
 + \Re \left\langle \QW u, (f_0^W)^2 u \right\rangle.
\end{aligned}
\end{equation}

We first estimate the second term of RHS in \eqref{P_0Rdcmp}, which can be decomposed as
\begin{equation}
\Re \left\langle \QW u, (f_0^W)^2 u \right\rangle
= \left\langle \QW f_0^W u, f_0^W u \right\rangle
  + \Re \left\langle [f_0^W I_N, \QW] u, f_0^W u \right\rangle.
\label{Rdcmp}
\end{equation}

Then, by the sharp G\aa rding inequality\begin{footnote}
{The sharp Gårding inequality for the matrix case can be proved by the anti-Wick quantization approach.}\end{footnote}, there exists a constant $C_0>0$, independent of $f_0$, such that
$$
 \langle \QW f_0^W u, f_0^W u \rangle
\geq - C_0 h \|f_0^W u\|^2.
$$

Next, we estimate the first term of RHS in~\eqref{P_0Rdcmp}.
Since the principal symbol\begin{footnote}
{The composition of a matrix-valued pseudodifferential operator and a scalar-valued pseudodifferential operator satisfies the same well-known formula as in the scalar case, since their symbols commute.}\end{footnote} $\frac{h}{i} \{ f_0 I_N, \Im P_0 \}$ of $[f_0^W I_N, \QW]$ is an anti-Hermitian matrix, and since $f_0$ is supported in $\{f_1 = 1\}$ by~\ref{item;D2},  the second term in~\eqref{Rdcmp} satisfies
$$
\Re \left\langle [f_0^W I_N, \QW] u, f_0^W u \right\rangle 
  \geq - \Order(h^2) \|f_1^W u\|^2 - \Order(h^\infty)\|u\|^2,
$$
and hence we obtain
$$
\Re \left\langle \QW u, (f_0^W)^2 u \right\rangle 
  \geq - C_0 h \|f_0^W u\|^2 - \Order(h^2)\|f_1^W u\|^2 - \Order(h^\infty)\|u\|^2. 
$$

Returning to \eqref{P_0Rdcmp} and dividing by $h$, we have
\begin{equation}\label{eq;3}
  \begin{split}
    h^{-1} \Im \left\langle \HW u, (f_0^W)^2 u \right\rangle 
    \geq & \frac{1}{2ih}\left\langle [(f_0^W)^2 I_N, \PW] u, u \right\rangle
    - C_0 \|f_0^W u\|^2 \\ & - \Order(h)\|f_1^W u\|^2 + \Order(h^\infty)\|u\|^2.
  \end{split}
\end{equation}

Since $P_1 = \Re P$, the operator $(2ih)^{-1}[(f_0^W)^2 I_N, \PW]$ has the principal symbol $f_0 \{\Re P_0, f_0 I_N\}$.
We define a matrix-valued symbol $G$ by
$$
G (x,\xi;h) = f_0 \big( \{\Re P_0,f_0 I_N\} + C^{1/2} \hat{f_0} P_0^*P_0 - (C_0 + C^{-1}) f_0 I_N \big) 
             + C a_0^2 I_N,
$$
and choose the constant $M$ in Remark~\ref{rem1} larger than $C_0$, which is determined by $\QW$.
Then $\{\Re P_0, f_0 I_N\} + C^{1/2} \hat{f_0} P_0^*P_0 - (C_0 + C^{-1}) f_0 I_N$ is nonnegative on the complement of $\Omega_0$ for sufficiently large $C>0$.

On the other hand, on $\Omega_0$, the function $\hat f_0$ is not nonnegative and $a_0 \equiv 1$ by~\ref{item;a0}.
Since $C$ is taken sufficiently larger than $C^{1/2}$,
the function $C^{1/2}\hat f_0 f_0 P_0^* P_0 + C a_0^2 I_N$, and hence $G$, is positive on $\Omega_0$.
Then, for sufficiently large $C>0$, the function $G$ is nonnegative everywhere.


By choosing $\chi(x,\xi) \in C_0^\infty(\mathbb{R}^{2n})$ such that
$
\supp(\det G) \subset \{\chi = 1\}, \; \supp \chi \subset \{ f_1 > 0 \},
$
we have
\[
\left\langle G^W u, u \right\rangle
= \left\langle G^W \chi^W u, \chi^W u \right\rangle
+ \mathcal{O}(h^\infty)\|u\|^2,
\]
since the symbol $G$ is supported in $\{ f_1 > 0 \}$ by~\ref{item;D2} and~\ref{item;a0}.
Therefore, applying the sharp G\aa rding inequality to the first term
on the right-hand side and since  $\|\chi^W u\| \leq C \|f_1^W u\| + \Order (h^\infty)$, we obtain
\[
\left\langle G^W u, u \right\rangle
\ge - C h \| f_1^W u \|^2 + \mathcal{O}(h^\infty)\|u\|^2.
\]

Then, by~\eqref{eq;3}, we have, modulo $\Order(h^\infty)\|u\|^2$,
\begin{equation*}
  \begin{split}
    h^{-1}\Im \left\langle \HW u, (f_0^W)^2 u \right\rangle
    \geq & - C^{1/2} \left\langle f_0^W \hat{f_0}^W \HW^* \HW u, u \right\rangle \\
    & - C \|a_0^W u\|^2 + \frac{1}{C} \|f_0^W u\|^2 - \Order(h)\|f_1^W u\|^2.
  \end{split}
\end{equation*}
From~\ref{item;D2} and the Cauchy--Schwarz inequality, we have
$$
h^{-1}\Im \left\langle \HW u, (f_0^W)^2 u \right\rangle
+ C^{1/2} \left\langle f_0^W \hat{f_0}^W \HW^* \HW u, u \right\rangle
\leq C h^{-1}\|f_1^W \HW u\| \| f_0^W u \|,
$$
for sufficiently small $h > 0$, and hence we have
$$
\|f_0^W u\|^2 \leq C(\|a_0^W u\|^2 + h \|f_1^W u\|^2 + h^{-1}\|f_1^W \HW u\| \| f_0^W u \|)   + \Order(h^\infty)\|u\|^2.
$$

Since, for any $C > 0$,
$$
h^{-1}\|f_1^W \HW u\| \| f_0^W u \| \leq C h^{-2}\|f_1^W \HW u\|^2 + C^{-1} \| f_0^W u \|^2,
$$
it follows that
$$
\|f_0^W u\|^2 \leq C(\|a_0^W u\|^2 + h \|f_1^W u\|^2 + h^{-2}\|f_1^W \HW u\|^2) + \Order(h^\infty)\|u\|^2,
$$
for sufficiently large $C > 0$.
This implies that the inequality~\eqref{eq;goal}:
$$
\|f_0^W u\| \leq C(\|a_0^W u\| + h^{1/2} \|f_1^W u\| + h^{-1}\|f_1^W \HW u\|) + \Order(h^\infty)\|u\|.
$$
This completes the proof of Proposition~\ref{propest}.

\end{proof}


\end{document}